\documentclass[pdflatex,sn-mathphys]{sn-jnl}

\usepackage{blindtext}
\usepackage{hyperref}
\usepackage{url}
\usepackage{enumerate}
\usepackage{makecell}
\usepackage{algpseudocode}

\newtheorem{Theorem}{\textbf{Theorem}}

\newcommand{\algorithmicinput}{\textbf{input}}
\newcommand{\INPUT}{\item[\algorithmicinput]}
\newcommand{\algorithmicoutput}{\textbf{output}}
\newcommand{\OUTPUT}{\item[\algorithmicoutput]}

\numberwithin{equation}{section} 
\numberwithin{Theorem}{section}
\newtheorem{Definition}{Definition}
\numberwithin{Definition}{section}

\numberwithin{Lemma}{section} 
\numberwithin{Proposition}{section}
\newtheorem{Example}{Example}

\numberwithin{Remark}{section}
\numberwithin{Example}{section}

\numberwithin{Observation}{section}

\numberwithin{Algorithm}{section}
\newtheorem{Note}{Note}

\numberwithin{Corollary}{section}

\title{A study on fuzzy plane and its application on fuzzy plane fitting}

\begin{document}
    
\author*[1]{\fnm{Diksha} \sur{Gupta}}\email{drdikshagupta9524@gmail.com}

\affil[1]{\orgdiv{School of Engineering,} \orgname{Ajeenkya DY Patil University}, \orgaddress{\city{Pune}, \postcode{412105}, \state{Maharashtra}, \country{India}}}

\abstract{
In this paper, I obtain an $S$-type fuzzy point when two fuzzy numbers for two independent variables and a corresponding fuzzy number for the dependent variable are given. A comprehensive study on a conceptualization of a fuzzy plane as a collection of fuzzy numbers, or fuzzy points is proposed. A perpendicular fuzzy distance from a fuzzy point to a fuzzy plane is also revisited. An application of the proposed fuzzy plane is made to fit a fuzzy plane to the available data sets of imprecise locations in $\mathbb{R}^3$. Moreover, a degree of fuzzily fitted fuzzy plane to the given data sets of fuzzy points is defined. All the fuzzy geometric construction and characteristics of fuzzy planes are explored with the help of same and inverse points ideas. All the study is supported by numerical examples and illustrated by fuzzy geometrical figures. This study provides a framework for developing a fuzzy plane-fitting model that will benefit the fields of curve detecting and fitting, image processing for industrial and scientific applications, signal processing, and problems of shape recognition.  }


\maketitle

\section{Introduction}
Tanaka et al. \cite{asai1982linear} originally introduced fuzzy linear regression since the problems with crisp input and fuzzy output commonly exist in real life. Tanaka's approach was used by many researchers \cite{chang1996applying, lai1994fuzzy, heshmaty1985fuzzy, wang2000insight} to minimize the total propagation of the fuzzy output. In \cite{diamond1988fuzzy}, a fuzzy least square method was introduced which took care of the minimization of the output total square error and applied in \cite{ming1997general, ruoning1997s}. 
The deficiency of Tanaka's approach \cite{asai1982linear} and fuzzy least square method \cite{diamond1988fuzzy} is discussed and modified in \cite{wang2000resolution}. A continued study has been made on fuzzy regression analysis \cite{tanaka1987fuzzy, tanaka1988possibilistic, tanaka1989possibilistic, kim1998evaluation}. In \cite{kao2002fuzzy}, a fuzzy linear regression model has been proposed with a better explanation that performs better than previous studies. All the forgoing study was revisited by Shapiro \cite{shapiro2005fuzzy} on fuzzy regression and discussed the issue related to Tanaka's approach and fuzzy least square methods. 
Pham \cite{pham2001representation} analyzed a need for fuzziness in computer-aided design \cite{pham1997hybrid, pham1999shape, pham2003fuzzy}.

In \cite{scarmana2016application}, the least square plane fitting interpolation process has been used for image reconstruction and enhancement. Vosselman and Dijkman \cite{vosselman20013d} described reconstruction-building models from point clouds and ground plans. Their method involved extracting planar faces from irregularly distributed point clouds using the 3D Hough transform.
 However, these models are extremely sensitive to noise. Hence, a modified study on these models has been done by Bodum \cite{bodum2004automatic}. After that, Kurdi et. al. \cite{tarsha2007hough} presented and compared the 3D-Hough transform and RANSAC algorithm for plane detection from Lidar data. 

A detailed research background and significance of point cloud has been highlighted in \cite{yang2020point, borrmann20113d} that is one of the fundamental measurement data sources and has been widely used in several fields such
as 3D-LiDAR \cite{yan2020online}, additive manufacturing \cite{xu2018plsp}, and unmanned aerial vehicle\cite{jian2018method}. A growing number of construction and infrastructure projects used point clouds to improve productivity, quality, and safety. Generally, airborne laser scanners are the most effective method for collecting rapid, high-density 3D data over cities. Once the 3D data were available, automatic data processing was the next task with the aim of constructing 3D building models. Therefore, automatic building modeling has been carried out once the roof planes are successfully detected. Borrmann et.al. \cite{borrmann20113d} investigated various forms of the Hough transform concerning their applicability to 
detect planes in 3D point clouds. A 3D classical Hough transform has been used to extract planar faces from the irregularly allocated point clouds \cite{vosselman20013d}. Drost and Ilic \cite{drost2015local}proposed a novel approach to detect and segment some 3D primitives, namely, planes, spheres, and cylinders from 3D point clouds. In \cite{yang2020point}, a new method based on a modified fuzzy C-means clustering algorithm with reserved feature information has been outlined. In \cite{mirzaei20223d}, three efficient applications of point clouds in the construction and infrastructure industries have been presented.

An ordinary least squares regression method to fit a plane was initially developed by Gauss \cite{gauss1809theoria}. This approach is easier to understand when two variables are involved, where $x$ is the independent variable, and $y$ is the dependent variable. However, it can also be applied to higher-dimensional models.
All forgoing study focuses on the theoretical point of view of regression analysis to give a relationship between a list of $n$ independent and a corresponding list of $n$ dependent fuzzy data. However, there is a lack of information about the fuzzy shapes which has been fitted to fuzzy data. A geometrical interpretation of fitted fuzzy functions such as fuzzy lines, fuzzy curves, fuzzy planes, etc., has not been mentioned. An explanation of the least square method in the geometrical sense to fit the imprecise data has not been given. Since a list of $n$ fuzzy numbers for independent variables and a corresponding list of $n$ fuzzy numbers for dependent variables generate the fuzzy points, a geometrical description of the fuzzy shape that has to be fitted to these data has not been described. Hence, we make a study on fuzzy plane fitting by following the classical plane fitting procedure using a fuzzy geometrical approach. According to our different approaches, we will take care that fuzzy plane fitting will not have an increasing spread. In our study, we extend the total least square approach in the fuzzy environment to get a fitted fuzzy plane. In ordinary least squares regression, all error is assumed to lie in the dependent dimension, and therefore the residual errors in that variable have been minimized.
There are many situations when using total least squares would be a preferable strategy in which residual errors are minimized in a direction perpendicular to the best-fit plane. The total least squares approach is not edge-cutting (Adcock, \cite{adcock1877note, adcock1878problem}; Pearson, \cite{pearson1901liii}), and according 
to Anderson \cite{anderson1984estimating}, it has been reinvented many times \cite{sturgul1970best, gower1973very}. An important advantage of using orthogonal regression is that the magnitudes of the residual errors are independent of the orientation of the 
best-fit plane concerning any of the reference axes (see \cite{fernandez2005obtaining}). In \cite{petravs2010total}, classical mathematical methods like orthogonal regression or errors in variable methods and a Matlab toolbox to solve these methods have been presented.

A few studies have been focused on fuzzy regression analysis by the fuzzy geometrical approach.
In recent work, \cite{ghosh2015general}, a generalized version of a fuzzy line has been proposed and applied to fit a fuzzy line to a dataset containing imprecise points or locations.
After that, in \cite{das2022conceptualizing}, Das and Chakraborty conceptualize a linear fuzzy function of a fuzzy line as a collection of fuzzy numbers or fuzzy points.

In regression analysis, a plane plays a significant role in analyzing the relationships between a dependent and two independent variables, in $\mathbb{R}^3$. However, a crisp plane is not always enough when observed data is inherently imprecise. When a description of data is vague, imprecise, or inadequate, their relationship by classical geometry may not be possible. An efficient presentation of the relation between objects or data with inherently imprecise realizations necessitates formulating a fuzzy plane. Hence, the classical definition of a plane has been extended to the fuzzy environment in $\mathbb{R}^3$, in \cite{ghosh2023analytical}. In the present study, I revisit the construction of fuzzy planes in which a linear fuzzy equation can be perceived as a fuzzy plane. Also, a fuzzy plane is conceptualized as a collection of fuzzy numbers or fuzzy points. Some real-life examples based on industry are explained to show the significance of the proposed fuzzy plane. The purpose is to provide a unified foundation of a framework for developing fuzzy geometric modeling which will benefit both computer vision applications and creative design.

The remainder of the paper is organized as follows. Preliminary definitions and notations are dealt with in section in Section \ref{sec2}. A fuzzy linear equation is represented as a fuzzy plane that is perceived as a group of fuzzy numbers, or space fuzzy points in Section \ref{sec3}. Section \ref{sec4} includes the formulation of a fuzzy plane fitted to the given imprecise locations. In Section \ref{sec4}, an application of the proposed fitted fuzzy plane is delineated. In Section \ref{sec5}, the article concludes and discusses possible research areas for the future.

\section{Preliminaries}\label{sec2}

In this paper, we refer \cite{ghosh2021analytical} for the definitions and terminology of $LR$-type fuzzy numbers \cite{ghosh2021analytical}, plane fuzzy points \cite{ghosh2021analytical}, $S$-type fuzzy points \cite{ghosh2021analytical}, and fuzzy distance between space fuzzy points \cite{ghosh2021analytical}. We refer \cite{ghosh2023analytical} for the definition of distance between a fuzzy plane and an $S$-type fuzzy point. Hence, we omit these definitions from the present study. We put a \emph{tilde bar} over  small or capital letters, i.e., $\widetilde{a}$, $\widetilde{b}$, $\widetilde{c}$, \ldots, $\widetilde{A}, \widetilde{B}, \widetilde{C}$, \ldots, to indicate fuzzy sets in $\mathbb{R}^{n}$. The notation $\mu\left(x \middle\lvert\widetilde{A}\right)$  represents the membership function of a fuzzy set $\widetilde{A}$ in $\mathbb{R}^n$, $x \in \mathbb{R}^n$.

The following definition deals with the notion of translation of an $S$-type fuzzy point along a direction $(\ell, m, n)\in \mathbb{R}^3$.

\begin{Definition}\label{translation}
\normalfont{(\emph{Translation of an $S$-type fuzzy point along a direction $(\ell, m, n)$}).\\
	 Let $\widetilde{P}(a, b, c)$ be an $S$-type fuzzy point (see Definition $3.2$ in \cite{ghosh2021analytical}) whose membership number at a point $(x, y, z)$ be evaluated by 
	\[\mu \left((x, y, z)\middle\lvert\widetilde P\right) = f(x-a, y - b, z - c). \]
	Let \[\tfrac{x-a}{\ell}=\tfrac{y-b}{m}=\tfrac{z-c}{n}=\lambda\] be a direction passing through $(a, b, c)$.
	 Then, the translation of $\widetilde{P}(a, b, c)$, say $\widetilde {P}_T$, along a direction $(\ell, m, n)$ is defined by the membership function as
	\[\mu \left((x, y, z)\middle\lvert\widetilde P_T\right) = f(x-(a+\lambda \ell), y - (b+\lambda m), z - (c+\lambda n)). \]

\noindent Explicitly, the translation of $(x, y, z)\in \widetilde{P}(0)$ along a direction $(\ell, m, n)$ to a new position $(x', y', z')$ is obtained by using the translation matrix
	$$ T=
	\begin{pmatrix}
	1 & 0 & 0 & \lambda\ell\\
	0 & 1 & 0 & \lambda m\\
	0 & 0 & 1 & \lambda n\\
	0 & 0 & 0 & 1
	\end{pmatrix}
	$$ such that $T(x, y, z)=(x', y', z')$.
}
	\end{Definition}
During translation, $\widetilde{P}(0)$ gets shifted to another position while keeping the shape and size of the fuzzy point intact. If $\widetilde{P}_T$ is obtained by the translation of $\widetilde{P}$ along a direction, then $\widetilde{P}$ and $\widetilde{P}_T$ are said to be translation copies of each other.

In the following subsection, we perceive that two fuzzy numbers for the independent variables and the corresponding fuzzy number for the dependent variable form an $S$-type fuzzy point. 

\subsection{Conceptualization  of an \texorpdfstring{$S$}{Lg}-type fuzzy point by two fuzzy numbers for independent variables and corresponding fuzzy number for a dependent variable}
Let $\widetilde{a}=(a-l_1/a/a+l_1)_{L_1R_1}$ and $\widetilde{b}=(b-l_2/b/b+l_2)_{L_2R_2}$ be two fuzzy numbers for independent variables, say $a$ and $b$, respectively, where $l_1$, $l_2$, be any number and $L_1=L_2=R_1=R_2=\max\{0, 1-\lvert x \lvert\}$. Let $\widetilde{c}=(c-l_3/c/c+l_3)_{L_3R_3}$ be a fuzzy number for a dependent variable $c$, where $l_3$ be any number, and $L_3=R_3=\max\{0, 1-\lvert x \lvert\}$. Then, a fuzzy set $\widetilde{P}$ at $(a, b, c)$ defined by 
\begin{equation}\label{eq1}
	\mu\left((x, y,z)\middle\lvert\widetilde{P}(a, b, c)\right)
	= \min\left\lbrace \mu\left(x \middle\lvert \widetilde{a} \right), \mu\left(y \middle\lvert \widetilde{b} \right), \mu\left(z \middle\lvert \widetilde{c} \right)\right\rbrace,
	\end{equation}
is a space fuzzy point.
One can note that the fuzzy set defined in (\ref{eq1}) qualifies all the properties of space fuzzy points (see Definition 3.2 in \cite{ghosh2021analytical}).

Now,
 
\begin{align*}
& ~ \min\left\lbrace \mu\left(x \middle\lvert \widetilde{a} \right), \mu\left(y \middle\lvert \widetilde{b} \right), \mu\left(z \middle\lvert \widetilde{c} \right)\right\rbrace\\
= & ~ 1-\max\left\lbrace 1- \mu\left(x \middle\lvert \widetilde{a} \right), 1- \mu\left(y \middle\lvert \widetilde{b} \right), 1- \mu\left(z \middle\lvert \widetilde{c} \right) \right\rbrace.\\
= & ~ 1-\max\left\lbrace \left\lvert\tfrac{x-a}{l_1}\right\lvert, \left\lvert\tfrac{y-b}{l_2}\right\lvert, \left\lvert\tfrac{z-c}{l_3}\right\lvert  \right\rbrace.
\end{align*}
Hence, we can see that a space fuzzy point $ \widetilde{P}(a, b, c)$ defined by (\ref{eq1}) is an $S$-type fuzzy point (see Definition 3.2 and Theorem 3.1 in \cite{ghosh2021analytical}).   

\begin{Note}
    \normalfont{It is noticeable that if we take one fuzzy number $\widetilde{a}$ for the independent variable and the corresponding fuzzy number $\widetilde{b} $ for a dependent variable in (\ref{eq1}), then, without loss of generality, the fuzzy set defined in (\ref{eq1}) reduces to a plane fuzzy point. }
\end{Note}

Figure \ref{fig1} interprets a geometrical visualization of a cuboid-shaped space fuzzy point when two fuzzy numbers $\widetilde{a}$, $\widetilde{b}$ along the lines parallel to the $x$ and $y$-axes, and a fuzzy number $\widetilde{c}$ along a line parallel to the $z$-axis are given. It is noted that to get a cuboid-shaped space fuzzy point, we need three reference fuzzy numbers, that is fuzzy numbers along the reference axes.

\begin{figure}[ht]
	\begin{center}
		\includegraphics[scale=1]{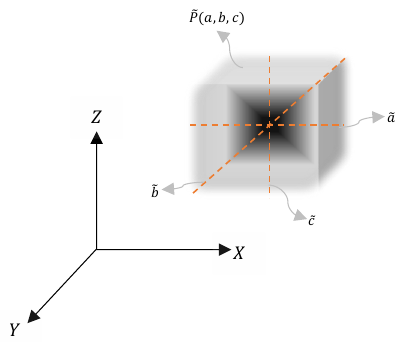}
		\caption{Cuboid-shaped space fuzzy point}
		\label{fig1}
	\end{center}
\end{figure}

\section{Perception of a fuzzy plane as a collection of fuzzy numbers and \texorpdfstring{$S$}{Lg}-type fuzzy points}\label{sec3}

In the literature, a variety of formulations of fuzzy planes have been demonstrated (see \cite{ghosh2023analytical}). To the best of my knowledge, a fuzzy equation that represents a fuzzy plane has not been explained previously. In the present study, I show that a fuzzy equation of the form $\widetilde{a}x+\widetilde{b}y+\widetilde{c}z+\widetilde{d}=0$ also represents a fuzzy plane, say $\widetilde{\varPi}$, where $\widetilde{a}$, $\widetilde{b}$, $\widetilde{c}$ and $\widetilde{d}$ are fuzzy parameters. A mathematical representation of a fuzzy plane $\widetilde{\varPi}: \widetilde{a}x+\widetilde{b}y+\widetilde{c}z+\widetilde{d}=0$ can be demonstrated as follows

\begin{align*}
&~ \mu\left((x, y, z)\middle\lvert\widetilde{\varPi}\right)
=\sup\{ \alpha: (x, y, z )\in \varPi^\alpha: a^\alpha x + b^\alpha y+ c^\alpha z + d^\alpha=0 \},
\end{align*}
where
\begin{equation}\label{eq5}
\begin{cases}
 a^\alpha, b^\alpha, c^\alpha, d^\alpha ~\text{are the same points of}~ \widetilde{a}, \widetilde{b}, \widetilde{c}, \widetilde{d} & \text{if $x, y, z \geq 0$},\\
  a^\alpha, b^\alpha, c^\alpha, d^\alpha ~\text{are the same points of}~ -\widetilde{a}, \widetilde{b}, \widetilde{c}, \widetilde{d} & \text{if $x\leq 0, y\geq 0, z \geq 0$},\\
   a^\alpha, b^\alpha, c^\alpha, d^\alpha ~\text{are the same points of}~ \widetilde{a}, -\widetilde{b}, \widetilde{c}, \widetilde{d} & \text{if $x\geq 0, y\leq 0, z \geq 0$},\\
    a^\alpha, b^\alpha, c^\alpha, d^\alpha ~\text{are the same points of}~ \widetilde{a}, \widetilde{b}, -\widetilde{c}, \widetilde{d} & \text{if $x\geq 0, y\geq 0, z \leq 0$},\\

     a^\alpha, b^\alpha, c^\alpha, d^\alpha ~\text{are the same points of}~ -\widetilde{a}, -\widetilde{b}, \widetilde{c}, \widetilde{d} & \text{if $x\leq 0, y\leq 0, z \geq 0$},\\
     
     a^\alpha, b^\alpha, c^\alpha, d^\alpha ~\text{are the same points of}~ \widetilde{a}, -\widetilde{b}, -\widetilde{c}, \widetilde{d} & \text{if $x\geq 0, y\leq 0, z \leq 0$},\\

     a^\alpha, b^\alpha, c^\alpha, d^\alpha ~\text{are the same points of}~ \widetilde{a}, -\widetilde{b}, -\widetilde{c}, \widetilde{d} & \text{if $x\leq 0, y\geq 0, z \leq 0$},\\
  
   a^\alpha, b^\alpha, c^\alpha, d^\alpha ~\text{are the same points of}~ \widetilde{a}, \widetilde{b}, \widetilde{c}, -\widetilde{d} & \text{if $x, y, z \leq 0$}.\\
\end{cases}
	\end{equation}
The Equation (\ref{eq5}) describes that a fuzzy plane is a collection of crisp planes in the $\mathbb{R}^3$-space with varied membership values.

According to the Dubois and Prade \cite{dubois1980fuzzy}, the equation $\widetilde{\varPi}: \widetilde{a}x+\widetilde{b}y+\widetilde{c}z+\widetilde{d}=0$ can be seen as a fuzzifying function $\widetilde{f}$ from $\mathbb{R}$ to $\widetilde{P}(Z)$, i.e., $\widetilde{f}:z \rightarrow \left( \tfrac{\widetilde{d}}{\widetilde{c}}-\tfrac{\widetilde{a}x}{\widetilde{c}}-\tfrac{\widetilde{b}y}{\widetilde{c}}\right)$, where $\widetilde{P}(Z)$ denotes the fuzzy power set of $Z$. Fuzzifying functions are defined as the sum-min compositions of their associated fuzzy function.
This representation can also be seen as $\widetilde{z}=\widetilde{f}(x, y)=\left( \tfrac{\widetilde{d}}{\widetilde{c}}-\tfrac{\widetilde{a}x}{\widetilde{c}}-\tfrac{\widetilde{b}y}{\widetilde{c}}\right)$. Hence, the equation $\widetilde{a}x+\widetilde{b}y+\widetilde{c}z+\widetilde{d}=0$ represents a collection of fuzzy numbers $\widetilde{z}$ at the points $(x,y, z)\in\mathbb{R}^3$ along the direction perpendicular to $xy$-plane. A fuzzy plane can be conceptualized as a collection of fuzzy numbers, represented as $((x, y, z), \widetilde{z})$, where $(x, y, z)$ is the core of $\widetilde{z}$.

Let $\widetilde{z}=\widetilde{f}(h, k)= \left( \tfrac{\widetilde{d}}{\widetilde{c}}-\tfrac{\widetilde{a}h}{\widetilde{c}}-\tfrac{\widetilde{b}k}{\widetilde{c}}\right)$ be a fuzzy number, placed along the line $x=h, y=k$ which is perpendicular to $xy$-plane. Then, the membership function of $\widetilde{z}$ can be defined as follows:
\begin{equation}\label{eq2}
	\mu\left((x, y, z) \middle\lvert\widetilde{f}(h, k)\right)
	= \begin{cases}
\mu\left( z \middle\lvert\widetilde{f}(h, k)\right) & \text{if $x=h, y=k$}\\
 0 & \text{otherwise.}
\end{cases}
	\end{equation}
By varying values of $(h, k)\in\mathbb{R}^2$, the equation $\widetilde{a}x+\widetilde{b}y+\widetilde{c}z+\widetilde{d}=0$ produces the collection of fuzzy numbers represented by Equation (\ref{eq2}). Easily it can be noted that the collection of fuzzy numbers by Equation (\ref{eq2}) is equivalent to the fuzzy plane obtained by equation $\widetilde{a}x+\widetilde{b}y+\widetilde{c}z+\widetilde{d}=0$.
Let $\left(x,y, \tfrac{d^\alpha}{c^\alpha}-\tfrac{a^\alpha x}{c^\alpha}-\tfrac{b^\alpha y}{c^\alpha}\right)$ be a point on $x=h, y=k$. Then, by Equation (\ref{eq2}),   

\begin{equation}\label{eq3}
	\mu\left(\left(x,y, \tfrac{d^\alpha}{c^\alpha}-\tfrac{a^\alpha x}{c^\alpha}-\tfrac{b^\alpha y}{c^\alpha}\right) \middle\lvert\widetilde{f}(h, k)\right)
	= \alpha.
	\end{equation}
So, the points on $\widetilde{f}(h, k)(0)$ with membership grade $\alpha$ lies on the plane $\varPi^\alpha:a^\alpha x+b^\alpha y+c^\alpha z+d^\alpha=0$. Equivalently, the plane $\varPi^\alpha: a^\alpha x+b^\alpha y+c^\alpha z+d^\alpha=0 $ includes all the points having membership number $\alpha$. Thus, the chunks of fuzzy numbers $\widetilde{f}(h, k)$ is equivalent to fuzzy plane  $\widetilde{\varPi}: \widetilde{a}x+\widetilde{b}y+\widetilde{c}z+\widetilde{d}=0$, for varying $(h, k)\in\mathbb{R}^2$. Hence, the fuzzy plane $\widetilde{\varPi}: \widetilde{a}x+\widetilde{b}y+\widetilde{c}z+\widetilde{d}=0$ can be seen as a $\underset{(x, y)}{\bigcup}\widetilde{f}(x, y)$, where $(x, y)\in\mathbb{R}^2$.

\begin{Note}
\normalfont{In the above scenario, I have taken the $xy$-plane as the reference plane and the lines parallel to the $z$-axis as reference lines along which the fuzzy numbers are defined by (\ref{eq2}). But, if the reference plane is any general plane $lx+my+nz+d=0$ and $\tfrac{x-p}{l}=\tfrac{y-q}{m}=\tfrac{z-r}{n}$ is a line perpendicular to the plane $lx+my+nz+d=0$, then the membership function of $\widetilde{f}(h,k)$, can be established as follows
\begin{equation}\label{eq4}
	\mu\left((h, k, z) \middle\lvert\widetilde{f}(h, k)\right)
	= \begin{cases}
\mu\left( \lambda n +r \middle\lvert\widetilde{f}(h, k)\right) & \text{if $\tfrac{h-p}{l}=\tfrac{k-q}{m}=\tfrac{z-r}{n}=\lambda$,}\\
 0 & \text{otherwise.}
\end{cases}
	\end{equation}  
 }
\end{Note}

\subsection{Fuzzy plane as a chunks of space fuzzy points}
In this subsection, we see a fuzzy plane $\widetilde{\varPi}: \widetilde{a}x+\widetilde{b}y+\widetilde{c}z+\widetilde{d}=0$ as a collection of fuzzy points.
For the sake of simplicity, we consider $\widetilde{d}=0$. Let a point $(x, y, z)$ be in the crisp plane $ax+by+cz=0$. Consider a line $y=mx, z=0$ passing through $(0, 0, 0)$ in $xy$-plane. Since, the equation $\widetilde{a}x+\widetilde{b}y+\widetilde{c}z=0$ can be perceived as a fuzzification function $\widetilde{f}(x, mx)=\tfrac{\left(\widetilde{a}+\widetilde{b}m\right) x}{-\widetilde{c}}$, where $\widetilde{a}$, $\widetilde{b}$, $\widetilde{c}$ are fuzzy parameters. So, at the point $(x, mx, 0)$, the  fuzzy number $\widetilde{z}=\widetilde{f}_m(x, mx)= \tfrac{\left(\widetilde{a}+\widetilde{b}m\right) x}{-\widetilde{c}}$ gives the fuzziness in the vertical dimension (perpendicular to $xy$-plane). Similarly, the fuzzy number $\widetilde{f'}_m (z)=\tfrac{-\widetilde{c}z}{\widetilde{a}+m\widetilde{b}}$ gives the fuzziness in the horizontal dimension which is deduced from $\widetilde{a}x+\widetilde{b}y+\widetilde{c}z=0$.
One can note that the order pair $(\widetilde{f'}_m (z), \widetilde{z})$ form a fuzzy point at the point $(x, mx, z)$ by (\ref{eq1}). Hence, for varying $(x, mx, z)$, a fuzzy plane can be conceptualized as a collection of fuzzy points $(\widetilde{f'}_m (z), \widetilde{z})$.\\
Also, for some $m\in\mathbb{R}$, these collections of fuzzy points form a space fuzzy line (see Section $3$ in \cite{ghosh2023analytical}), say $\widetilde{L}_m$, represented as $\widetilde{L}_m=\underset{(x, mx, z)}{\bigcup}(\widetilde{f'}_m (z), \widetilde{z})$ on $\widetilde{\varPi}$, whose core is the line $y=m x$. The fuzzy plane $\widetilde{\varPi}$ can be seen as the union of the fuzzy lines $\widetilde{L}_m$, for varying $m\in\mathbb{R}$.  

Next, with the help of the intercept form of a fuzzy plane as defined in \cite{ghosh2023analytical}, I give a numerical example that shows a fuzzy plane as a collection of fuzzy numbers and fuzzy points.
 
\begin{Example}
\normalfont{Consider three fuzzy numbers $ \tilde{a}=\tilde{b}=\tilde{c}=(-2/1/2) $. It can be easily seen that the points `$ 3 \alpha-2 $, $3 \alpha-2$ and $3 \alpha-2 $', or `$ 2-\alpha$, $2-\alpha$ and $ 2-\alpha $' appear on $ \tilde{a}$, $\tilde{b} $ and $ \tilde{c} $ are the same points, respectively, with the membership number $ \alpha\in[0, 1] $.
		The support of $\widetilde{\varPi}_{I} $ is the group of crisp planes satisfying the points $ (3 \alpha-2, 0, 0)$, $(0, 3 \alpha-2, 0)$ and $(0, 0, 3 \alpha-2) $, or $ (2-\alpha, 0, 0)$, $(0, 2-\alpha, 0)$ and $(0, 0, 2-\alpha) $ of $ \tilde{a}(0)$, $\tilde{b}(0) $ and $ \tilde{c}(0) $, respectively.

  Mathematically,\\
{\footnotesize{
 \begin{align*} 
		 &~\widetilde{\varPi}_{I}(0)
		 =\bigcup_{\alpha\in[0,1]}\left\lbrace  (x,y,z): \dfrac{x}{3 \alpha-2}+\dfrac{y}{3 \alpha-2}+\dfrac{z}{3 \alpha-2}=1, \text{or} ~ \dfrac{x}{2-\alpha}+\dfrac{y}{2-\alpha}+\dfrac{z}{2-\alpha}=1 \right\rbrace\\ 
			\end{align*}
       }
}and $ \widetilde{\varPi}_{I}(1): x+y+z=1 $.

  The $\alpha$-level planes $x+y+z=3 \alpha-2$ or $x+y+z=2-\alpha$ can be written as $z=3 \alpha-2-h-k$ or $z=2-\alpha-h-k$. Now, for varying $\alpha\in [0, 1]$, a fuzzy number $\widetilde{z}=\widetilde{1}-h-k$ will be obtain whose points with same membership value $\alpha$ are $z=3 \alpha-2-h-k$ and $z=2-\alpha-h-k$ along the line $x=h, y=k$. Consequently, a fuzzy plane $\widetilde{\varPi}_{I} $ can be portrayed as a chunk of fuzzy numbers $\widetilde{z}=\widetilde{1}-h-k$ for varying points along the line $x=h, y=k$.  
  Similarly, a fuzzy number $\widetilde{f'}_{mc}(z)= \tfrac{\widetilde{1}-z-c}{(m+1)}$ is obtained along the direction perpendicular to the $x=h, y=k$, where $h$ and $k$ satisfies the relation $ k=m h+c$, for $m, c\in\mathbb{R}$. The order pair $(\widetilde{f'}_{m c}(z), \widetilde{z})$ form a fuzzy point, say $\widetilde{P}(h, m h+c, z)$, at the point $(h, m h+ c, z)$. By (\ref{eq1}), the membership number of $\widetilde{P}(h, m h+c, z)$ can be established as 
  \begin{equation}\label{eq6}
	\mu\left((x, y,z)\middle\lvert\widetilde{P}(h, m h+c, z)\right)
	= \min\left\lbrace \mu\left(x \middle\lvert \widetilde{f'}_{m c} \right), \mu\left(y \middle\lvert m\widetilde{ f'}_{m c}+c \right), \mu\left(z \middle\lvert \widetilde{z} \right)\right\rbrace.
	\end{equation}
 Here, after a simple calculation, we get $\widetilde{z}=(-9/-6/-5)$ and $\widetilde{f'}_{m c}=(1.5/3/3.5)$, for $h=3$, $k=4$ and $m=c=1$.
 Let us suppose the membership number of a point $(3, 4, -7)$ has to be evaluated.
 So, \begin{align*}
	\mu\left((3, 4,-7)\middle\lvert\widetilde{P}(3, 4, -6)\right)
	= &~\min\left\lbrace \mu\left(3 \middle\lvert \widetilde{f'}_{m c} \right), \mu\left(4 \middle\lvert \widetilde{f'}_{m c}+1 \right), \mu\left(-7 \middle\lvert \widetilde{z} \right)\right\rbrace\\
 = & ~ \min \left\lbrace 1, 1, 0.667 \right\rbrace \\
=& ~ 0.667.
	\end{align*}
  }
\end{Example}

Let us go through the study of the distance between a fuzzy plane and an $S$-type fuzzy point which is an important tool to get a fitted fuzzy plane for further study. We define two types of distance from an $S$-type fuzzy point to a fuzzy plane as follows.

\subsection{Fuzzy distance between an \texorpdfstring{$S$}{Lg}-type fuzzy point and a fuzzy plane}

\begin{Definition}\label{distance1}\normalfont{(\emph{Vertical distance between a fuzzy plane and an $S$-type fuzzy point}).
Let $ \widetilde{\varPi} $ be a fuzzy plane. Consider $xy$-plane to be a core plane of $\widetilde{\varPi}$ and the $z$-axis to be perpendicular to the plane. Let $\widetilde{P}$ be an $S$-type fuzzy point whose core lies on $z$-axis. If we call the distance function $d((x_1, y_1, z_1), (x_2, y_2, z_2))=\lvert z_1-z_2\lvert$, then the vertical distance between $\widetilde{P}$ and $\widetilde{\varPi}$, denoted as $\widetilde{D}_v$, can be established by
\begin{equation}\label{Pdistance1}
    \widetilde{D}_v=\underset{\alpha\in[0,1]}{\bigvee}\left[d_{l}^\alpha, d_{u}^\alpha\right],
    \end{equation}
where
        \[d_{l}^\alpha=\inf\left\{d(p,q): ~p\in\widetilde{P}(\alpha) ~\text{and}~ q\in\widetilde{\varPi}(\alpha)\right\}\] and \[d_{u}^\alpha=\sup\left\{d(p,q): ~p\in\widetilde{P}(\alpha) ~\text{and}~ q\in\widetilde{\varPi}(\alpha)\right\}.\]
The notion $d_{l}^\alpha$ and $d_{u}^\alpha$ are associated with the membership number $\alpha$.

}
\end{Definition}

\begin{Definition}\label{distance2}\normalfont{(\emph{Perpendicular distance between a fuzzy plane and an $S$-type fuzzy point}).
Let $\widetilde{P}$ be an $S$-type fuzzy point and let $ \widetilde{\varPi} $ be a fuzzy plane. The perpendicular distance between $\widetilde{P}$ and $\widetilde{\varPi}$, say $\widetilde{D}_p$, can be established by
\begin{equation}\label{Pdistance2}
    \widetilde{D}_p=\underset{\alpha\in[0,1]}{\bigvee}\left[d_{l}^\alpha, d_{u}^\alpha\right],
    \end{equation}
where
        \[d_{l}^\alpha=\inf\left\{d(p,\varPi): ~p\in\widetilde{P}(\alpha) ~\text{and}~ \varPi\in\widetilde{\varPi}(\alpha)\right\}\] and \[d_{u}^\alpha=\sup\left\{d(p,\varPi): ~p\in\widetilde{P}(\alpha) ~\text{and}~ \varPi\in\widetilde{\varPi}(\alpha)\right\},\] where $d(p,\varPi)$ is the perpendicular distance between the point $p$ and the plane $\varPi$.
The notion $d_{l}^\alpha$ and $d_{u}^\alpha$ are associated with the membership number $\alpha$.

}
\end{Definition}

\begin{Theorem}\label{dis}
The distance $\widetilde{D}$ between a fuzzy plane and an $S$-type fuzzy point is a fuzzy number.
\end{Theorem}
\begin{proof}
    Similar to the proof of Theorem $4.1$ in \cite{ghosh2021analytical}.
\end{proof}

\begin{Note}
\normalfont{A numerical example that illustrates the characteristic of the distance from an $S$-type fuzzy point to a fuzzy plane can be seen in the same scenario as explained in Example $5.2$ in \cite{ghosh2023analytical}.

}
\end{Note}

In the next section, we obtain an orientation from a set of fuzzy points consisting of finding a fitted fuzzy plane through the proposed analysis of the fuzzy plane.

\section{Fuzzy plane fitting}\label{sec4}

In the following subsection, we fit a fuzzy plane to the available data of imprecise locations in $\mathbb{R}^3$. It can be seen as an application of the fuzzy plane interpreted by $\widetilde{a}x+\widetilde{b}y+\widetilde{c}z+\widetilde{d}=0$.

\subsection{Formulation of a fuzzy plane fitted to the given \texorpdfstring{$S$}{Lg}-type fuzzy points}
Let $\widetilde{P}_1$, $\widetilde{P}_2$, $\cdots$, $\widetilde{P}_n$ be the given data to whom we want to fit a fuzzy plane such that the perpendicular distances between available space fuzzy points and the fitted fuzzy plane is minimum.

The scenario to fit a fuzzy plane is to first find the plane $ax+by+cz+d=0$ that fits the core of the fuzzy points $\widetilde{P}_1$, $\widetilde{P}_2$, $\cdots$, $\widetilde{P}_n$. Then, perpendicularly translate the fuzzy points to $ax+by+cz+d=0$. We try to find the best-fitted fuzzy plane, which has the minimum possible value for the sum of the square of the perpendicular distances of the fuzzy points. First, to understand what is meant by a perpendicular shifting of a fuzzy point, let us define it.

\begin{Definition}\normalfont{(\emph{Perpendicular shifting of an $S$-type fuzzy point}).\\
Let $\widetilde{P}(a, b, c)$ be an $S$-type fuzzy point with the membership function \[\mu \left((x, y, z)\middle\lvert\widetilde P\right) = f(x-a, y - b, z - c). \]
Let \[\tfrac{x-a}{\ell}=\tfrac{y-b}{m}=\tfrac{z-c}{n}=\lambda\] be a line passing through $(a, b, c)$ perpendicularly to the plane $ax+by+cz+d=0$.
Then, the perpendicularly shifted fuzzy point, say $\widetilde {P'}(a+\lambda \ell, b+\lambda m, c+\lambda n)$, for some $\lambda$, along a direction $(\ell, m, n)$ is established by the membership function as
	\[\mu \left((x, y, z)\middle\lvert\widetilde {P}(a+\lambda \ell, b+\lambda m, c+\lambda n)\right) = f(x-(a+\lambda \ell), y - (b+\lambda m), z - (c+\lambda n)). \]
    }
\end{Definition}

Let us now find a fuzzy plane $\widetilde{\varPi}_F: \widetilde{a}x+\widetilde{b}y+\widetilde{c}z+\widetilde{d}=0$, that fits to the given fuzzy points $\widetilde{P}_1 (a_1, b_1, c_1), \widetilde{P}_2 (a_2, b_2, c_2), \cdots, \widetilde{P}_n (a_n, b_n, c_n)$. Let $\widetilde{P'}_1 (a_1+\lambda \ell, b_1+\lambda m, c_1+\lambda n), \widetilde{P'}_2 (a_2+\lambda \ell, b_2+\lambda m, c_2+\lambda n), \cdots, \widetilde{P'}_n (a_n+\lambda \ell, b_n+\lambda m, c_n+\lambda n)$, for some $\lambda\in \mathbb{R}$, be perpendicularly shifted fuzzy points to  $\widetilde{\varPi}_F(1)$.
Consider the fuzzy numbers $\widetilde{N}_1, \widetilde{N}_2, \cdots, \widetilde{N}_n$ along the  lines perpendicular to $ax+by+cz+d=0$ from the $(a_1, b_1, c_1), (a_2, b_2, c_2), \cdots, (a_n, b_n, c_n)$ on $\widetilde{P'}_1(0), \widetilde{P'}_2(0), \cdots, \widetilde{P'}_n(0)$, respectively.
The following Algorithm \ref{algo1} illustrates the formulation of $\widetilde{\varPi}_F$.

\begin{algorithm}[H]
	\begin{algorithmic}

		\INPUT{Given $n$ continuous fuzzy numbers $\widetilde{N}_1, \widetilde{N}_2, \cdots, \widetilde{N}_n$.}
		\For{each $\alpha$ from $0$ to $1$}
  
  \State {Compute the same points $u^{\alpha}_1\in \widetilde{N}_1(0), u^{\alpha}_2\in\widetilde{N}_2(0), \cdots, u^{\alpha}_n\in\widetilde{N}_n(0) $ and $v^{\alpha}_1\in \widetilde{N}_1(0), v^{\alpha}_2\in\widetilde{N}_2(0), \cdots, v^{\alpha}_n\in\widetilde{N}_n(0) $. }
		\State Fit the planes $\varPi_1: a_{1}^\alpha x+b_{1}^\alpha y+c_{1}^\alpha z+d_{1}^\alpha=0$ and $\varPi_2: a_{2}^\alpha x+b_{2}^\alpha y+c_{2}^\alpha z+d_{2}^\alpha=0$ that approximate the points $u^{\alpha}_1, u^{\alpha}_2, \cdots, u^{\alpha}_n$ and $v^{\alpha}_1, v^{\alpha}_2, \cdots, v^{\alpha}_n$, respectively, by the method of the total least squares (see \cite{petravs2010total}).
  \State Returns $\widetilde{a}=[a_{1}^\alpha, a_{2}^\alpha]$, $\widetilde{b}=[b_{1}^\alpha, b_{2}^\alpha]$, $\widetilde{c}=[c_{1}^\alpha, c_{2}^\alpha]$ and $\widetilde{d}=[d_{1}^\alpha, d_{2}^\alpha]$
		\OUTPUT{$\widetilde{\varPi}_F=\underset{\alpha\in [0, 1]}{\bigvee} \widetilde{\varPi}_F(\alpha)$.}
		
  \EndFor
	\end{algorithmic}
	\caption{To evaluate $\widetilde{\varPi}_F$}
	\label{algo1}
\end{algorithm}

\begin{Theorem}\label{thm4.1}
    Let $\widetilde{P}_1, \widetilde{P}_2, \cdots, \widetilde{P}_n$ be the given imprecise locations in $\mathbb{R}^3$. The fuzzy plane $\widetilde{\varPi}_F$ obtained from the Algorithm \ref{algo1} is always unique.
\end{Theorem}
\begin{proof}
    Since, in the classical regression analysis, a unique curve is fitted to the data points by the curve fitting method, hence, by the steps given in Algorithm \ref{algo1}, it is trivial that the fuzzy plane $\widetilde{\varPi}_F$ fits the given imprecise data uniquely.
\end{proof}

\begin{Theorem}\label{thm4.2}
  Let $\widetilde{P}_1, \widetilde{P}_2, \cdots, \widetilde{P}_n$ be the given imprecise locations in $\mathbb{R}^3$. A fuzzy plane is being fitted to them by the Algorithm \ref{algo1}. The sum $\widetilde{D}_{p_1}^2+\widetilde{D}_{p_2}^2+\cdots+\widetilde{D}_{p_n}^2$ will be minimum amongst all possible fuzzy plane fits passing through $\widetilde{P'}_1, \widetilde{P'}_2, \cdots, \widetilde{P'}_n$. Here, $\widetilde{D}_{p_i}$ is the perpendicular distance between $\widetilde{P}_i$ and $\widetilde{\varPi}_F$, for $i=1, 2, \cdots, n$.
\end{Theorem}

\begin{proof}
    The proof is trivial from the Theorem \ref{thm4.1}.
\end{proof}

Let us say the surfaces $g_1(x, y, z)=0$ and $g_2(x, y, z)=0$ approximate the boundary points of the $\widetilde{N}_1(0)$, $\widetilde{N}_2(0)$, $\widetilde{N}_3(0), \cdots, \widetilde{N}_n(0)$. Thus, two boundary surfaces bound a fuzzy plane and there must be one crisp surface between those surfaces. Consequently, a fuzzy plane can be expressed by an equational form 
\begin{equation}\label{eqq}
 \left(g_1(x, y, z)/ax+by+cz+d/g_2(x, y, z) \right)_{LR}=0,
 \end{equation}
 where $L$ and $R$ be suitable reference function.
Here, the Equation (\ref{eqq}) means that the membership grade of the points on $\widetilde{\varPi}_{F}(0)$ gradually increases from 0 to 1 as we move from $g_1(x, y, z)=0$ to $ax+by+cz+d=0$ and decreases from 1 to 0 as we move from  $ax+by+cz+d=0$ to $g_2(x, y, z)=0$.

In the following, we give some examples that show the importance of the fuzzy plane fitted to the given imprecise data.

\subsection{An application of fitted fuzzy plane \texorpdfstring{$\widetilde{\varPi}_F$}{Lg}}
\begin{Example}\label{Ex1}\normalfont{(\emph{Fitted fuzzy plane $\widetilde{\varPi}_{F}$}).\\
    A company CEO wants to find a fuzzy function that estimates the goodness of revenue per year based on the number of products sold per year and the number of employees working that year.
 Let the goodness of revenue for the four years be represented by the following fuzzy numbers, say $\widetilde{N_1}$, $\widetilde{N_2}$, $\widetilde{N_3}$, $\widetilde{N_4}$, evaluated as
    
    \begin{equation}
    \mu \left(N \middle\lvert\widetilde{N_1}\right) =
    \begin{dcases*}
\max\lbrace 0, 1-\left(\tfrac{10-N}{10}\right) \rbrace & if $ 0\leq N\leq 10 $\\
 1 & if $N> 10$,\\
 \end{dcases*}
 \end{equation}

  {\footnotesize{\begin{equation}
    \mu \left(N \middle\lvert\widetilde{N_2}\right)=\mu \left(N \middle\lvert\widetilde{N_3}\right)=\mu \left(N \middle\lvert\widetilde{N_4}\right) =
    \begin{dcases*}
\max\lbrace 0, 1-\left(\tfrac{10-N}{10}\right)^2 \rbrace & if $ 0\leq N\leq 10 $\\
 1 & if $N> 10$.\\
 \end{dcases*}
 \end{equation}}}
  Let the following data be collected from the past
$(50, 20, \widetilde{N}_1)$, $(30, 5, \widetilde{N}_2)$, $(35, 20, \widetilde{N}_3)$, $(60, 25, \widetilde{N}_4)$. Let the same points of $\widetilde{N_1}$, $\widetilde{N_2}$, $\widetilde{N_3}$ and $\widetilde{N_4}$ be $10\alpha$, $10(1-\sqrt{1-\alpha})$, $10(1-\sqrt{1-\alpha})$, $10(1-\sqrt{1-\alpha})$.
By Algorithm \ref{algo1},

\begin{equation*}
\mu \left((x, y, z) \middle\lvert\widetilde{\varPi}_F\right)=
 \begin{dcases*}
 r & if   $ 0 \leq z \leq 10 $\\
 1 & if  $ z > 10 $,
 \end{dcases*}
 \end{equation*}
 where 
 \begin{align*}
 & r = \sup \{\alpha: (x, y, z) ~\text{lies on the planes fitted to}~ (50, 20, N^{\alpha}_1), (30, 5, N^{\alpha}_2),  \\
 ~&~~~~~(35, 20, N^{\alpha}_3),~ (60, 25, N^{\alpha}_4), ~\text{where}~ N^{\alpha}_1, N^{\alpha}_2, N^{\alpha}_3, N^{\alpha}_4 ~ \text{are same points}\\
 ~& ~~~~\text{of}~\widetilde{N}_1, \widetilde{N}_2, \widetilde{N}_3, \widetilde{N}_4, ~\text{for}~ \alpha\in[0, 1] \}.
 \end{align*}

The core plane $\widetilde{\varPi}_{F}(1): z=10$. For $\alpha=0.5$, $\widetilde{\varPi}_{F}(0.5): z=0.031 x-0.0178 y + 2.3694$. For $\alpha=0.8$, $\widetilde{\varPi}_{F}(0.8): z = 0.0375x - 0.0205y + 4.8631$. For $\alpha=0$, the boundary surface of the fuzzy plane is $z=0$. A geometrical representation of the fitted fuzzy plane for Example \ref{Ex1} is depicted in Figure \ref{fig2}.}
\end{Example}

\begin{figure}[ht]
	\begin{center}
		\includegraphics[scale=.4]{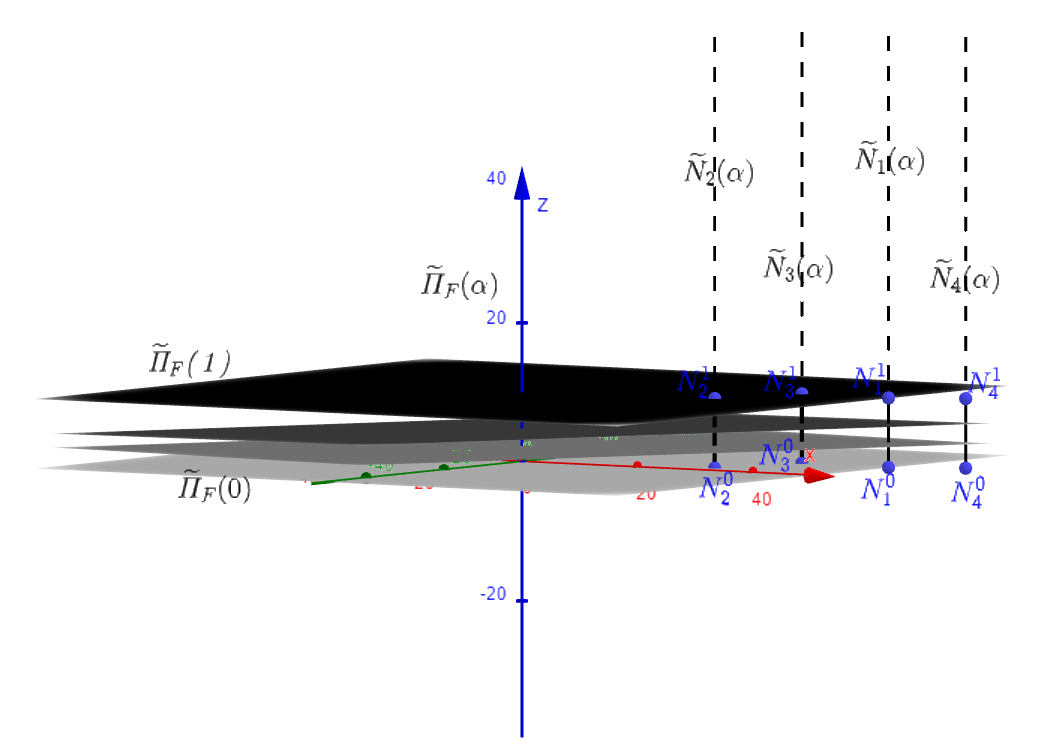}
		\caption{Fitted fuzzy plane for the Example \ref{Ex1}}
		\label{fig2}
	\end{center}
\end{figure}

The following example explains the construction of a fitted fuzzy plane when the data given in $x$, $y$, and $z$ directions are imprecise. 

\begin{Example}\label{Ex2}\normalfont{(\emph{Fitted fuzzy plane $\widetilde{\varPi}_{F}$}).\\
City planning officials want to obtain a fuzzy linear function that predicts the cost of the project based on the length of the construction project and the size of the construction firm (stories). Let the size, length, and cost be represented by fuzzy numbers as shown in $(a)$, $(b)$ and $(c)$ in Figure \ref{fig3}, respectively.

\begin{figure}[H]
 \begin{center}
     \begin{subfigure}{0.32\textwidth}
        \includegraphics[width=0.9\linewidth]{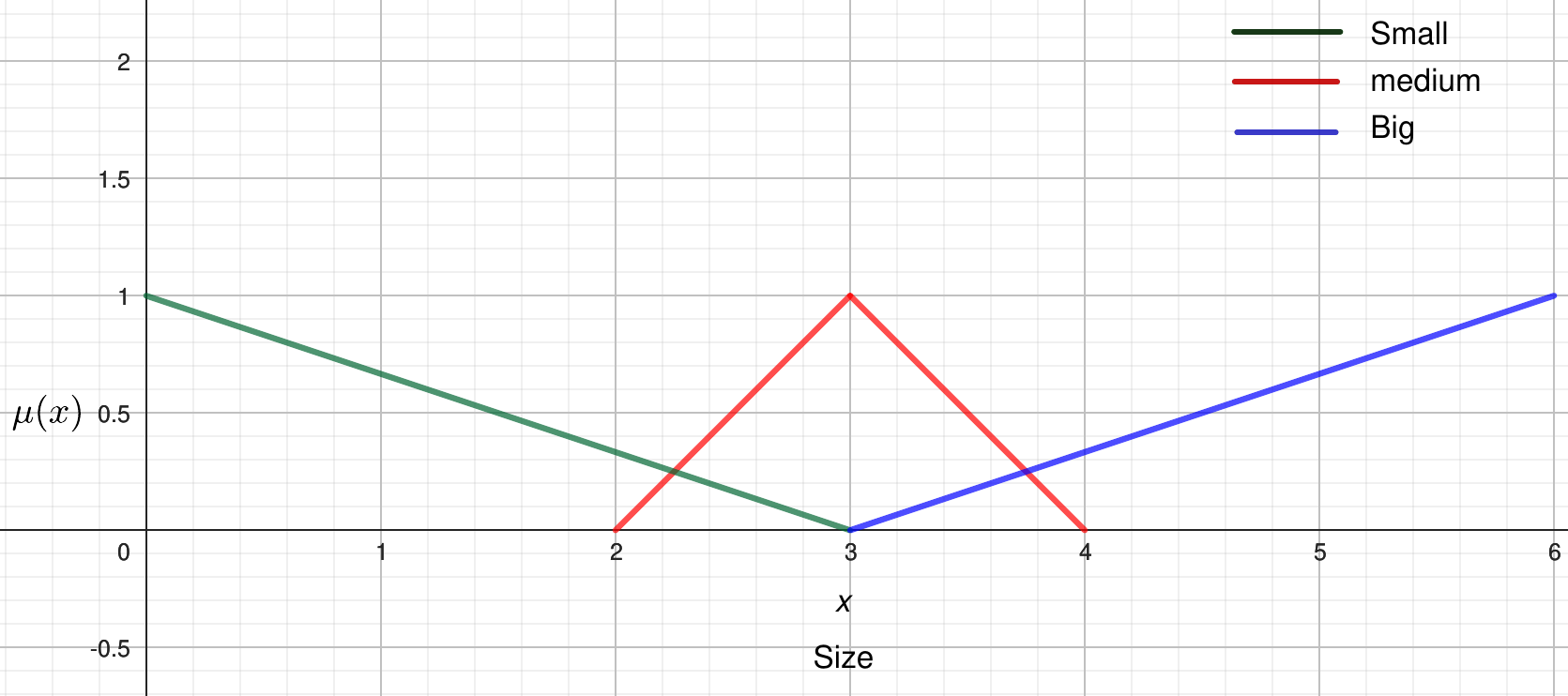}
        \caption{}
    \end{subfigure}
    \begin{subfigure}{0.32\textwidth}
        \includegraphics[width=0.9\linewidth]{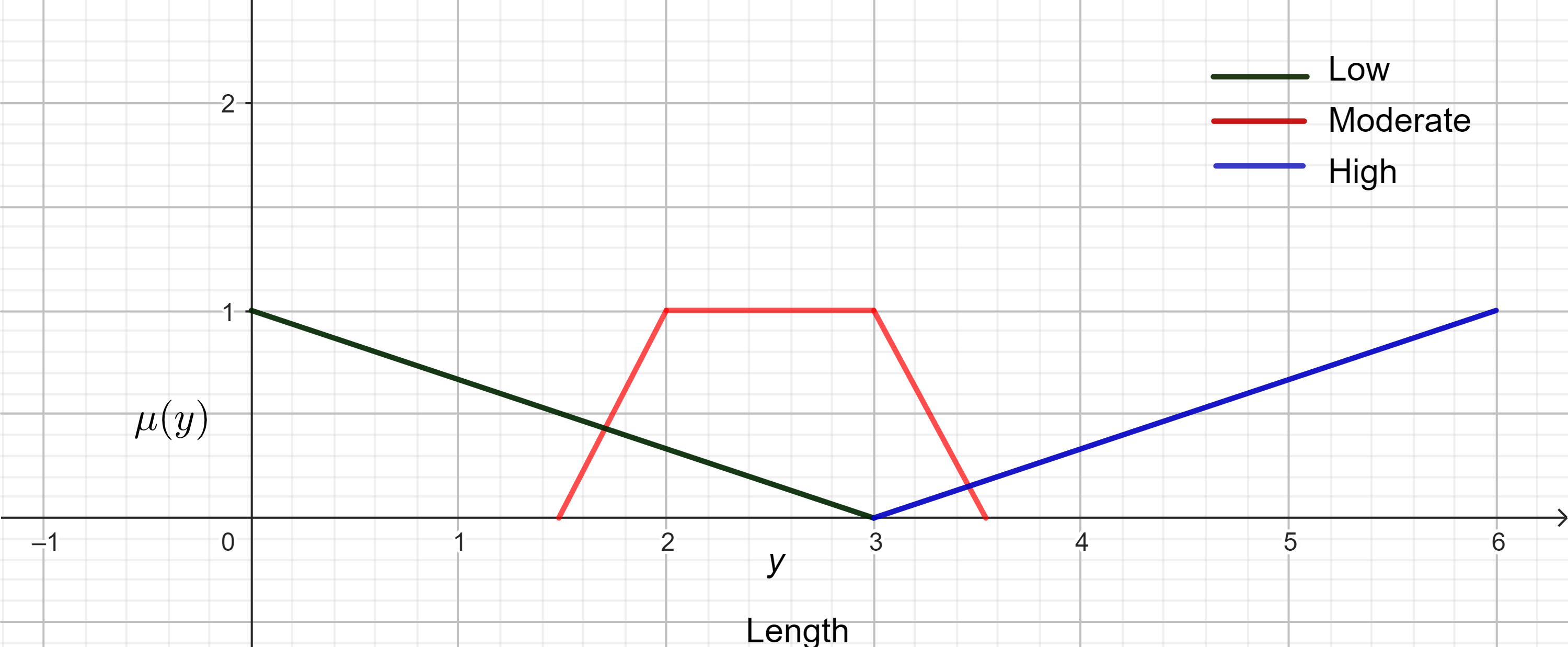}
        \caption{}
    \end{subfigure}
    \begin{subfigure}{0.32\textwidth}
        \includegraphics[width=0.9\linewidth]{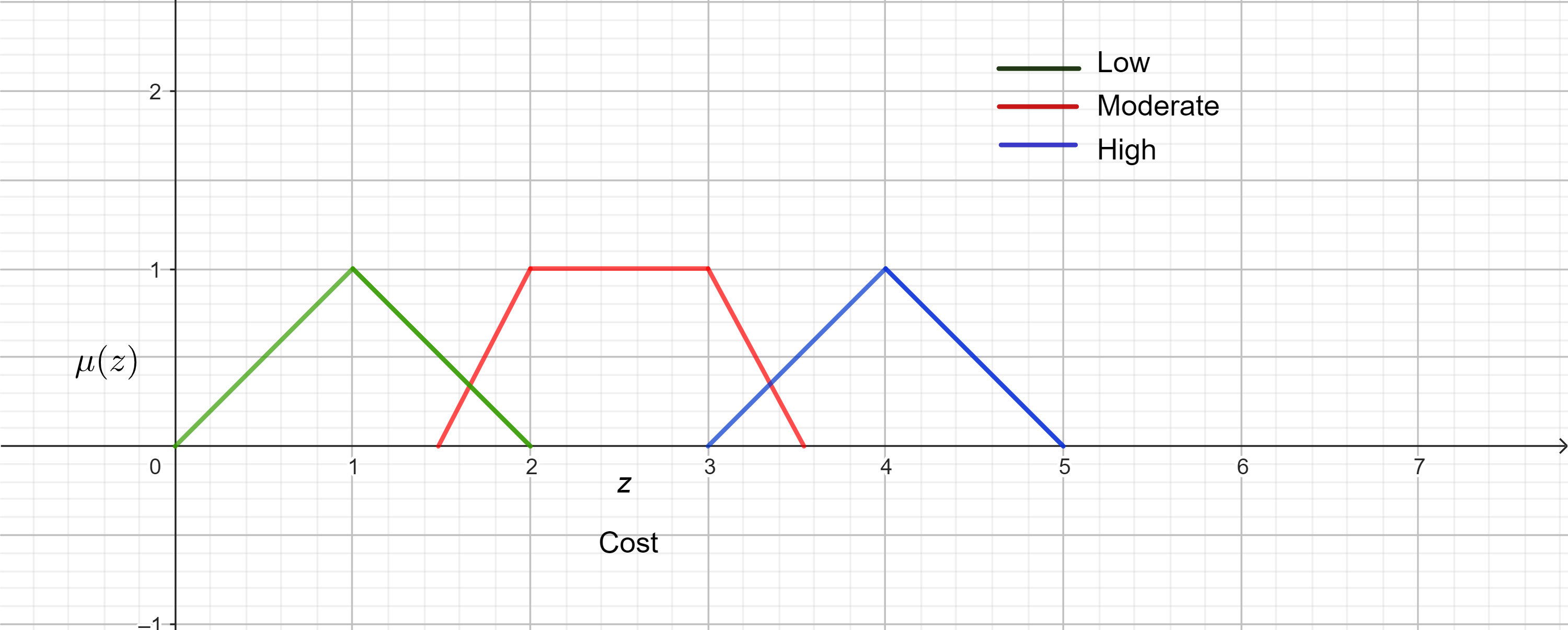}
        \caption{}
    \end{subfigure}
    \end{center}
 
    \caption{Representation of size, length, and cost as fuzzy numbers}
       \label{fig3}
\end{figure}  

Let the data collected from the past be $\left(\widetilde{x}_1, \widetilde{y}_1, \widetilde{z}_1\right)$, $\left(\widetilde{x}_2, \widetilde{y}_2, \widetilde{z}_2\right)$, $\left(\widetilde{x}_3, \widetilde{y}_3, \widetilde{z}_3\right)$, $\left(\widetilde{x}_4, \widetilde{y}_4, \widetilde{z}_4\right)$, say.

By Equation (\ref{eq1}), $\widetilde{P}_1=:\left(\widetilde{x}_1, \widetilde{y}_1, \widetilde{z}_1\right)$, $\widetilde{P}_2=:\left(\widetilde{x}_2, \widetilde{y}_2, \widetilde{z}_2\right)$, $\widetilde{P}_3=:\left(\widetilde{x}_3, \widetilde{y}_3, \widetilde{z}_3\right)$, and $\widetilde{P}_4=:\left(\widetilde{x}_4, \widetilde{y}_4, \widetilde{z}_4\right)$ are fuzzy points with core $(x_1, y_1, z_1)$, $(x_2, y_2, z_2)$, $(x_3, y_3, z_3)$, and $(x_4, y_4, z_4)$, respectively (see Figure \ref{fig4}). Let the membership functions of $\widetilde{P}_i$'s be 
\begin{equation}
    \mu \left((x, y, z)\middle\lvert\widetilde {P}_i\right) = 1-\max\left\lbrace \left\lvert\tfrac{x-x_i}{l_1}\right\lvert, \left\lvert\tfrac{y-y_i}{l_2}\right\lvert, \left\lvert\tfrac{z-z_i}{l_3}\right\lvert  \right\rbrace.
\end{equation}

Consider $ax+by+cz+d=0$ is a plane that fits the core points $(x_1, y_1, z_1)$, $(x_2, y_2, z_2)$, $(x_3, y_3, z_3)$, and $(x_4, y_4, z_4)$ by total least square method (see \cite{petravs2010total}).\\
Now, we perpendicularly shift the fuzzy points to the plane $ax+by+cz+d=0$. To do this, consider lines $L_i: \tfrac{x-x_i}{a}=\tfrac{x-y_i}{b}=\tfrac{z-z_i}{c}=\lambda$. The value of $\lambda$ for which the points $x'_i=x_i+\lambda a$, $y'_i=y_i+\lambda b$, $z'_i=z_i+\lambda c$ satisfy the plane $ax+by+cz+d=0$ gives the core points of shifted fuzzy points,  $\widetilde{P'}_i$ say. By the Definition \ref{translation}, the membership value of the fuzzy points $\widetilde{P'}_i$ is
\begin{equation}
    \mu \left((x, y, z)\middle\lvert\widetilde {P'}_i\right) = 1-\max\left\lbrace \left\lvert\tfrac{x-x'_i}{l_1}\right\lvert, \left\lvert\tfrac{y-y'_i}{l_2}\right\lvert, \left\lvert\tfrac{z-z'_i}{l_3}\right\lvert  \right\rbrace.
\end{equation}\\
Next, our attempt will be to find the fuzzy numbers on $\widetilde{P'}_1(0)$, $\widetilde{P'}_2(0)$, $\widetilde{P'}_3(0)$, $\widetilde{P'}_4(0)$ along the directions perpendicular to the $ax+by+cz+d=0$. Let $L_i:\tfrac{x-x_i}{a}=\tfrac{x-y_i}{b}=\tfrac{z-z_i}{c}$ be lines perpendicular to the $ax+by+cz+d=0$, for $i=1,2,3,4$. Consider a reference point $A$ on $\widetilde{P'}_i(0)$, for some $i=1, 2, 3, 4$. Let $\hat{u}$, $\hat{v}$, and $\hat{w}=\hat{u}\times\hat{v}$ be unit vectors along the edges from $A$. For a point $P=:(x, y, z)$ along the line $L_i$ to be inside the support of $\widetilde{P'}_i$, the vector $\Vec{AP}$ must have its coordinates in the reference $(\hat{u}$, $\hat{v}, \hat{w})$ contemporaneously on the limit of the range $[- 2 l_1, 2 l_2]$. That is, the equation  
\begin{equation}\label{seq}
\begin{dcases*}
 - 2 l_1 \leq \Vec{AP} \cdot \hat{u} \leq 2 l_1\\
 - 2 l_1 \leq \Vec{AP} \cdot \hat{v} \leq 2 l_1 \\
  - 2 l_1 \leq \Vec{AP} \cdot \hat{w} \leq 2 l_1, \\
 \end{dcases*}
\end{equation}
must satisfy contemporaneously.
The point $P$ will not lie inside $\widetilde{P'}_i(0)$ if the Equation (\ref{seq}) is not satisfied contemporaneously. Let $P_1$, $P_2$, $Q_1$, $Q_2$, $R_1$, $R_2$, and $S_1$, $S_2$ be the points that satisfy the Equation (\ref{seq}) are the intersecting points of the line $L_i$ and $\widetilde{P'}_i(0)$, respectively. These intersecting points give the boundary points of the supports of the fuzzy number $\widetilde{N}_i$ along $L_i$, respectively. Let us denote $N_i=(x'_i, y'_i, z'_i)$, then the fuzzy numbers $\widetilde{N}_i$ can be represented as $\widetilde{N}_1=\left(P_1/N_1/P_2\right)$, $\widetilde{N}_2=\left(Q_1/N_2/Q_2\right)$, $\widetilde{N}_3=\left(R_1/N_3/R_2\right)$, $\widetilde{N}_4=\left(S_1/N_4/S_2\right)$ along $L_i$ on the supports of $\widetilde{P'}_i (0)$, for $i=1,2,3,4$ (see Figure \ref{fig4}). 
Now, by the Algorithm \ref{algo1}, I get $\widetilde{a}=[a_{1}^\alpha, a_{2}^\alpha]$, $\widetilde{b}=[b_{1}^\alpha, b_{2}^\alpha]$, $\widetilde{c}=[c_{1}^\alpha, c_{2}^\alpha]$ and $\widetilde{d}=[d_{1}^\alpha, d_{2}^\alpha]$. Hence, the boundary surfaces of the fitted fuzzy plane $\widetilde{\varPi}_{F}$ are $a_{1}^\alpha x+ b_{1}^\alpha y+c_{1}^\alpha z+d_{1}^\alpha=0$ and $a_{2}^\alpha x+ b_{2}^\alpha y+c_{2}^\alpha z+d_{2}^\alpha=0$. The fuzzy plane 
$\widetilde{\varPi}_{F}$ can be expressed by an equation form 
\begin{equation}\label{eqe}
 \left(a_{1}^\alpha x+ b_{1}^\alpha y+c_{1}^\alpha z+d_{1}^\alpha/ax+by+cz+d/a_{2}^\alpha x+ b_{2}^\alpha y+c_{2}^\alpha z+d_{2}^\alpha \right)_{LR}=0.
 \end{equation}
}
\end{Example}

\begin{figure}[ht]
	\begin{center}
		\includegraphics[scale=.8]{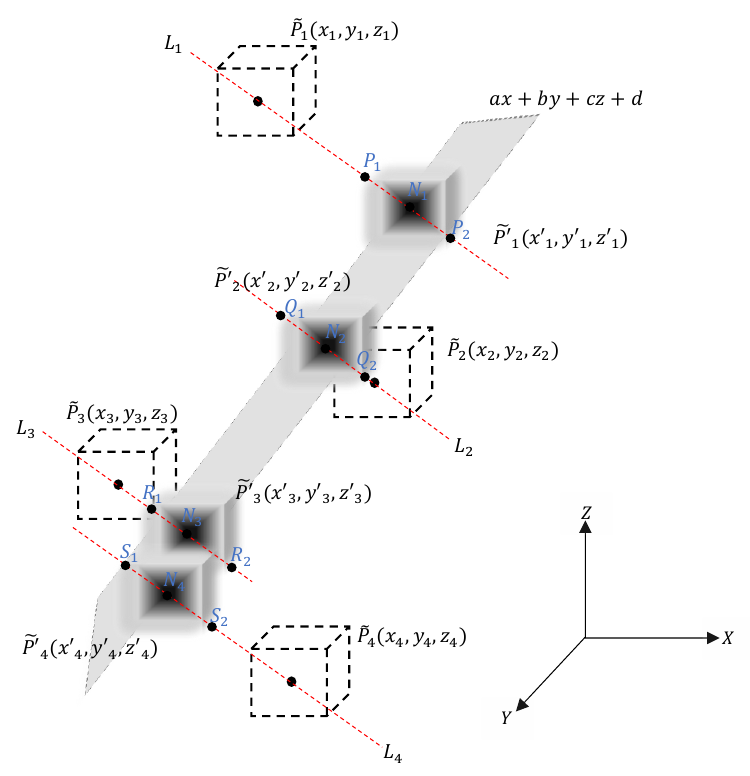}
		\caption{Fitted fuzzy plane for the Example \ref{Ex2}}
		\label{fig4}
	\end{center}
\end{figure}


\subsection{A degree of fuzzily fitted fuzzy plane to the given \texorpdfstring{$S$}{Lg}-type fuzzy points}

First, I will discuss the measurement of the belongingness of space fuzzy points to the fitted fuzzy plane. Then, with the help of that, I calculate the degree to which a fuzzy plane fits the given space fuzzy points. A high degree of fit is accompanied by good orientation quality. In the following, I define the containment of a space fuzzy point in a fuzzy plane and the degree of fit of the fuzzy plane.

\begin{Definition}\normalfont{(\emph{Containment of an $S$-type fuzzy point in a fuzzy plane}).\\
Let $\widetilde{\varPi}_{F}\equiv \left( g_1(x, y, z)/ax+by+cz+d/g_2(x, y,z)\right)_{LR}=0$ be a fuzzy plane fitted to space fuzzy points $\widetilde{P}_1$, $\widetilde{P}_2$, $\cdots$,  $\widetilde{P}_n$. The core points of $\widetilde{P}_1$, $\widetilde{P}_2$, $\cdots$, $\widetilde{P}_n$ lies on $\widetilde{\varPi}_{F}(1)$. The containment of $\widetilde{P}_i$ in the $\widetilde{\varPi}_{F}$ is obtained by a number, $\gamma$ say, where $\gamma$ is evaluated by 
\begin{equation}
\gamma
 = \begin{dcases*}
 1 & if $\widetilde{P}_{i}(0)\subset \widetilde{\varPi}_{F}(0)$,\\
 \gamma_1 & if $ \widetilde{P}_{i}(0)$ exceeds on $\widetilde{\varPi}_{F}(0)$ on the side of $g_1$,\\
  \gamma_2 & if $ \widetilde{P}_{i}(0)$ exceeds on $\widetilde{\varPi}_{F}(0)$ on the side of $g_2$,\\
 \min\{\gamma_1, \gamma_2\} & if $ \widetilde{P}_{i}(0)$ exceeds $\widetilde{\varPi}_{F}(0)$ on the either sides of $\widetilde{\varPi}_{F}(1)$,\\
 \end{dcases*}
\end{equation}
where $\gamma_1=\underset{(x,y,z): g_1(x, y, z)=0}{\sup} \mu\left((x, y, z)\middle\lvert\widetilde{P}_i\right)$ and $\gamma_2=\underset{(x,y,z): g_2(x, y, z)=0}{\sup} \mu\left((x, y, z)\middle\lvert\widetilde{P}_i\right)$}, for $i=1,2, \cdots, n$.
\end{Definition}

\begin{Note}
\normalfont{Note that for a containment of a space fuzzy point to the fitted fuzzy plane,  $\widetilde{P}_i(1)$ must belong to $ \widetilde{\varPi}_{F}(1)$. Otherwise, a space fuzzy point is not fuzzily contained in $\widetilde{\varPi}_{F}$, i.e., $\gamma=0$. }  
\end{Note}

In classical geometry, the points fitted by a plane may lie on the plane or not. In contrast, in the fuzzy environment, space fuzzy points may be fuzzily contained in the fitted fuzzy plane. In other words, there may be different cases such as $\widetilde{P}_{i}(0)\subset \widetilde{\varPi}_{F}(0)$, or  $ \widetilde{P}_{i}(0)$ exceeds on $\widetilde{\varPi}_{F}(0)$ on the side of $g_1$, or $ \widetilde{P}_{i}(0)$ exceeds on $\widetilde{\varPi}_{F}(0)$ on the side of $g_2$, or $ \widetilde{P}_{i}(0)$ exceeds $\widetilde{\varPi}_{F}(0)$ on the either sides of $\widetilde{\varPi}_{F}(1)$, or $\widetilde{P}_{i}(0)\bigcap \widetilde{\varPi}_{F}(0)$. Hence, I calculate the degree of a fuzzily fitted fuzzy plane to the given space fuzzy points.

\begin{Definition}\normalfont{(\emph{A degree of fuzzily fitted fuzzy plane to the given space fuzzy points}).\\
Let $\widetilde{\varPi}_{F}\equiv \left( g_1(x, y, z)/ax+by+cz+d/g_2(x, y,z)\right)_{LR}=0$ be a fuzzy plane fitted to space fuzzy points $\widetilde{P}_1$, $\widetilde{P}_2$, $\cdots$,  $\widetilde{P}_n$. The core points of $\widetilde{P}_1$, $\widetilde{P}_2$, $\cdots$, $\widetilde{P}_n$ lies on $\widetilde{\varPi}_{F}(1)$. A degree of fuzzily fitted fuzzy plane $\widetilde{\varPi}_{F}$ to the given space fuzzy points $\widetilde{P}_i$ is obtained by a number, $\delta$ say, where $\delta$ is evaluated by 
\begin{equation}\label{dff}
\delta=\min\{\gamma_1, \gamma_2, \cdots \gamma_n \}, 
\end{equation}
where $\gamma_i$ is the containment of $\widetilde{P}_i$ in the $\widetilde{\varPi}_{F}$, for $i=1,2, \cdots, n$.}
\end{Definition}

\begin{Note}
\normalfont{Note that if the core of space fuzzy points $\widetilde{P}_i$ does not belong to $ \widetilde{\varPi}_{F}(1)$, then the degree of a fuzzily fitted fuzzy plane to the given $\widetilde{P}_i$'s will always be zero, i.e., $\delta=0$ (obvious from Equation (\ref{dff})). }  
\end{Note}

\section{Conclusion}\label{sec5}
In this paper, firstly, I have made an explanation to perceive two fuzzy numbers for independent variables, and a corresponding fuzzy number for a dependent variable as an $S$-type fuzzy point. Next, my study has been based on a conceptualization of a fuzzy plane as a union of fuzzy numbers, or fuzzy points. A fuzzy plane which is visualized by a fuzzy linear equation has been approximated as a collection of fuzzy numbers, or fuzzy points. Fuzzy planes have also been proposed in equational form. A perpendicular fuzzy distance from an $S$-type fuzzy point to a fuzzy plane has also been revisited. I have given a different approach to fit a fuzzy plane to the available sets of imprecise locations in $\mathbb{R}^3$ using fuzzy geometry. The proposed fitted fuzzy plane has been deployed to find a linear pattern of given imprecise data. In the proposed study, I have given some Examples \ref{Ex1}, \ref{Ex2} that explain the significance of the proposed fuzzy plane. Moreover, a degree of fuzzily fitted fuzzy plane to the given data sets of space fuzzy points has been defined. 

As per the approach of \cite{das2022conceptualizing}, it can be noticed that the space fuzzy point obtained on the fuzzy plane may not fit within the fuzzy plane. A detailed study to obtain a geometrical transformation of a space fuzzy point to fit within the given fuzzy plane will be presented in the future. In fuzzy linear inequality, fuzzy regression, and trend analysis of fuzzy data, a fitted fuzzy plane to fuzzy data may have many potential applications. Further study can be focused on a formulation of other types of fuzzy shapes and their applications in fuzzy regression models to find the vague relationship between dependent and independent variables. I will concentrate my future efforts on the applications of the proposed study on creative design and computer vision applications.

\section*{Acknowledgments}
 The author acknowledges the support and resources provided by the School of Engineering at Ajeenkya DY Patil University, Pune, Maharashtra.

\section*{Compliance with Ethical Standards}
\begin{itemize}
    \item Conflict of interest: The authors declare no conflict of interest.
    \item Ethical approval: This article does not contain any studies with human participants or animals performed by any of the authors.
    \item Availability of data and materials: NA
    \item Funding statement: NA
\end{itemize}

\bibliography{sn-bibliography}


\end{document}